\documentclass[10pt]{amsart}
\usepackage{amsmath,mathrsfs,amssymb,cite}

\usepackage{graphicx,wasysym}
\usepackage[small]{caption}
\usepackage[colorlinks=true,urlcolor=blue,
citecolor=red,linkcolor=blue,linktocpage,pdfpagelabels,bookmarksnumbered,bookmarksopen]{hyperref}
\usepackage[english]{babel}

\usepackage[left=2.63cm,right=2.63cm,top=2.6cm,bottom=2.6cm]{geometry}

\newcommand{\R}{{\mathbb R}}

\numberwithin{equation}{section}
\newtheorem{theorem}{Theorem}[section]
\newtheorem{proposition}[theorem]{Proposition}

\newtheorem{conject}[theorem]{Conjecture}

\newtheorem{problem}[theorem]{Problem}
\theoremstyle{definition}

\newcommand{\lapl}{\Delta}

\newcommand{\iu}{{\rm i}}

\title[Bifurcation for quasi-linear Schr\"odinger equations]{On 
a bifurcation value related \\ to quasi-linear Schr\"odinger equations}

\author{Marco Caliari}
\address{Dipartimento di Informatica
\newline\indent
Universit\`a degli Studi di Verona
\newline\indent
C\'a Vignal 2, Strada Le Grazie 15
\newline\indent
I-37134 Verona, Italy}
\email{marco.caliari@univr.it}
\email{marco.squassina@univr.it}

\author{Marco Squassina}

\thanks{Research supported by 2009 PRIN: {\em Metodi Variazionali e Topologici
nello Studio di Fenomeni non Lineari}}

\begin{document}
	

\subjclass[2010]{35A15; 35B06; 74G65; 35B65}

\keywords{Bifurcation phenomena, minimization problems, quasi-linear Schr\"odinger equation.}

\begin{abstract}
By virtue of numerical arguments we study a bifurcation phenomenon occurring for a class of
minimization problems associated with the quasi-linear Schr\"odinger equation.
\end{abstract}

\maketitle

\section{Introduction}
Various physical situations are described by quasi-linear equations of the form
\begin{equation}
   \label{eq.schr1}
 \begin{cases}
 \iu\phi_t+\lapl\phi+\phi\lapl |\phi|^2+|\phi|^{p-1}\phi=0 & \text{in $(0,\infty)\times \R^3$},\\
 \phi(0,x)=\phi_0(x) & \text{in $\R^3$},
 \end{cases}
\end{equation}
where $1<p<11$, $\iu$ stands for the imaginary unit and the unknown $\phi:(0,\infty)\times \R^3\to {\mathbb C}$ is a complex valued
function. For example, it is used in plasma physics and
fluid mechanics, in the theory of Heisenberg ferromagnets and magnons and in condensed
matter theory. See e.g.\ the bibliography of~ \cite{CJS}. Motivated by the classical stability results
of the semi-linear Schr\"odinger equation
\begin{equation}
   \label{eq.schrr}
 \begin{cases}
 \iu\phi_t+\lapl\phi+|\phi|^{p-1}\phi=0 & \text{in $(0,\infty)\times \R^3$},\\
 \phi(0,x)=\phi_0(x) & \text{in $\R^3$},
 \end{cases}
\end{equation}
namely the stability of ground states (least energy solutions of $-\Delta u+\omega u =|u|^{p-1}u$, $\omega>0$) for
$1 < p < \frac{7}{3}$ and their instability for $p \geq \frac{7}{3}$  (see \cite{BeCa,CaLi}),
an interesting and physically relevant question for equation \eqref{eq.schr1} is the orbital stability
of ground state solutions of 
\begin{equation}\label{gs1}
   -\Delta u-u\Delta u^2+\omega u =|u|^{p-1}u\quad  \text{in $\R^3$}.
\end{equation}
When $1<p <\frac{13}{3}$, it is conjectured in \cite{CJS} that the ground states
are orbitally stable. However, in \cite{CJS} this result was not proved. 
Instead, it was considered the stability issue for
the minimizers of the problem
\begin{equation}
   \label{defvalc}
\mathscr{M}(c)=\inf_{\substack{u\in X \\ \|u\|_{L^2(\R^3)}^2 = c}} {\mathscr E}(u),
\end{equation}
where $1<p<\frac{13}{3}$ and the energy functional ${\mathscr E}$ defined on 
$X=\big\{u\in H^1(\R^3):\,  u|Du| \in L^2(\R^3)\big\}$ by
\begin{equation}
\label{eq:energyfunc}
{\mathscr E}(u) = \frac{1}{2}\int_{\R^3} (1+2u^2)|Du|^2dx- \frac{1}{p+1}\int_{\R^3} |u|^{p+1}dx.
\end{equation}
This problem, which looks interesting by itself, can be seen as a useful tool for a first
attempt towards the understanding of orbital stability of ground states of~\eqref{gs1} for
fixed $\omega >0$. Denoting by $\mathcal{G}(c)$ the set of
solutions to \eqref{defvalc}, in \cite{CJS} the authors prove that
if $1<p<\frac{13}{3}$ and $c>0$ is such that $m(c)<0$, then ${\mathcal G}(c)$ is orbitally stable (see \cite{CJS}
for the definition). Concerning \eqref{defvalc}, we learn \cite{CJS} that the following facts hold:

\begin{proposition}[CJS, \cite{CJS}]
   \label{raccolta} 
The following properties hold.
   \begin{enumerate}
       \item Assume that $1< p < \frac{13}{3}$. Then $\mathscr{M}(c) > - \infty$ for every $c>0$. 
           \vskip4pt
           \item Assume that $p=\frac{13}{3}$. Then $\mathscr{M}(c) > - \infty$ for $c>0$ small and $\mathscr{M}(c)=- \infty$ for $c>0$ large.
       \vskip4pt
       \item Assume that $\frac{13}{3} < p <11$. Then $\mathscr{M}(c)=-\infty $ for every $c>0$.
      \vskip4pt
       \item Assume that $1< p <\frac{13}{3}$ and let $c >0$ be such that
$\mathscr{M}(c)<0$. Then problem~\eqref{defvalc} admits a positive minimizer $u_c\in X$ which is
radially symmetric and radially decreasing. Moreover any solution $u_c \in X$ of~\eqref{defvalc} satisfies
\begin{equation}
	\label{eq:stationary}
- \Delta u_c- u_c\Delta u_c^2 +\lambda_c u_c=|u_c|^{p-1}u_c
\end{equation}
for some $\lambda_c>0$.
\vskip4pt
\item Assume that $1< p <\frac{7}{3}$. Then $\mathscr{M}(c) <0$ for every $c>0$.
   \vskip4pt
   \item Assume that $\frac{7}{3} \leq  p \leq \frac{13}{3}$. Then $\mathscr{M}(c) \leq 0$ for every $c>0$. 
\vskip4pt
\item
Assume that $\frac{7}{3} \leq p <\frac{13}{3}$. Then there exists 
$$
c_\sharp=c(p)>0
$$
such that:
\begin{enumerate}
\item If $c < c_\sharp$ then  $\mathscr{M}(c) =0$  and  $\mathscr{M}(c)$ {\bf does not admit} a minimizer.
\vskip4pt
\item If $c > c_\sharp$ then $\mathscr{M}(c)<0$  and $\mathscr{M}(c)$ {\bf admits} a minimizer. 
\end{enumerate}
\end{enumerate}
\end{proposition}

The motivation for the formulation of the following problems is mainly related to point (7) in the statement of Proposition \ref{raccolta}.
Notice that the value of $c_\sharp (p)$ can be characterized as follows
$$
c_\sharp(p)=\inf\{c>0:\, \mathscr{M}(c)<0\},
$$ 
This could help while trying to numerically compute the value of $c_\sharp(p)$.
The bifurcation value $c_\sharp(p)$ which appears in the previous Proposition \ref{raccolta} is obtained in \cite{CJS} by
an indirect argument by contradiction and thus it is not explicitly available for calculation
through a given formula. Hence, on these basis, it seems natural to formulate the following problems:

\begin{problem}
	\label{bounds}
Provide some lower and upper bounds of $\mathscr{M}(c)$ for $c>0$ and $p\in [\frac{7}{3},\frac{13}{3})$.
\end{problem}

\begin{problem}
	\label{bifproblem}
Numerically compute or provide bounds for the map $c_\sharp:[\frac{7}{3},\frac{13}{3})\to (0,+\infty)$.
\end{problem}

\begin{problem}
	\label{gsproblem}
Numerically compute the solutions to $\mathscr{M}(c)$ for $c>c_\sharp(p)$ with $p\in [\frac{7}{3},\frac{13}{3})$.
\end{problem}

\noindent
For the corresponding, more classical \cite{CaLi}, semi-linear minimization problem
	\begin{equation*}
	\mathscr{M}_{{\rm sl}}(c)=\inf_{\substack{u\in H^1(\R^3) \\ \|u\|_{L^2(\R^3)}^2 = c}} {\mathscr E}_{{\rm sl}}(u),\qquad
	{\mathscr E}_{{\rm sl}}(u) = \frac{1}{2}\int_{\R^3}|Du|^2dx- \frac{1}{p+1}\int_{\R^3} |u|^{p+1}dx,\quad 1<p<\frac{7}{3},
	\end{equation*}
there is no bifurcation phenomena, namely $c_\sharp(p)=0$ for every $1<p<\frac{7}{3}$
and Problem \ref{gsproblem} was studied in \cite{CS} by arguing on a suitable associated parabolic problem 
in order to decrease initial energies computed on Gaussian initial guesses.  
	An important point both for analytical and numerical purposes is the fact that
	minimizers of $\mathscr{M}_{{\rm sl}}(c)$ have fixed sign, are radially symmetric, decreasing and 
	unique, up to translations and multiplications by $\pm 1$. In principle,
	the uniqueness is used in \cite{CS} to justify that the numerical algorithm really provides
	the solution to the minimization problem $\mathscr{M}_{{\rm sl}}(c)$.
	To show the uniqueness of solutions to $\mathscr{M}_{{\rm sl}}(c)$ one can argue as follows. By the result of~\cite{kwong}, 
	for any $\lambda>0$ there exists a unique (up to translations)
	positive and radially symmetric solution $r=r_\lambda:\R^3\to\R$ of
	\begin{equation*}
	-\Delta r+\lambda r=r^{p}\qquad\text{in $\R^3$},
	\end{equation*}
	In turn,
	given $\lambda_1,\lambda_2>0$, if $r_1,r_2:\R^3\to\R$ denote, respectively, positive radial
	solutions of 
	\begin{equation}
		\label{dueeq}
	-\Delta r_1+\lambda_1 r_1=r_1^{p}\qquad\text{in $\R^3$},\qquad\,\,\,
	-\Delta r_2+\lambda_2 r_2=r_2^{p}\qquad\text{in $\R^3$},
	\end{equation}
	there exists some point $\xi\in\R^3$ such that
	\begin{equation}
		\label{realvscomp}
	r_2(x+\xi)=\mu r_1(\gamma x),\qquad \gamma:=\Big(\frac{\lambda_2}{\lambda_1}\Big)^\frac{1}{2},\quad
	\mu:=\Big(\frac{\lambda_2}{\lambda_1}\Big)^\frac{1}{p-1}.
	\end{equation}
	Let now $r_1$ and $r_2$ be two given solutions 
	to the minimization problem $\mathscr{M}_{{\rm sl}}(c)$ and let $-\lambda_1$ and $-\lambda_2$ be the
	corresponding Lagrange multipliers. By virtue of \cite[Theorem II.1, ii)]{CaLi}, $\lambda_1,\lambda_2>0$
	and $r_1,r_2$ are $C^2$ solutions of \eqref{dueeq}, are radially symmetric, radially decreasing and with fixed sign. In turn, 
	up to multiplication by $\pm 1$, $r_1,r_2>0$ and by \eqref{realvscomp} it holds
	$$
	c=\|r_2\|_{L^2(\R^3)}^2=\|r_2(\cdot+\xi)\|_{L^2(\R^3)}^2=\mu^2\gamma^{-3}\int_{\R^3} r_1^2(x)dx=c\mu^2\gamma^{-3},
	$$
	yielding in turn $\mu^2\gamma^{-3}=1$. By the definition of $\gamma$ 
	and $\mu$ in~\eqref{realvscomp}, we get $\lambda_1=\lambda_2$ and $\gamma=\mu=1$,
	yielding from~\eqref{realvscomp} as desired $r_1=r_2$, up to a translation.
	
	On the contrary, it is not currently known that
	minimizers $r\in X$ of the quasi-linear minimization problem~\eqref{defvalc} 
	are unique, up to translations and multiplications by $\pm 1$, 
	although it is conjectured that this is the case. If 
	$r\in X$ is a given minimizer for~\eqref{defvalc}, arguing as in \cite{CJS} it is possible to prove that
	it has fixed sign, so that, up to multiplication by $-1$, we may assume $r>0$. Then, $r$ is radially
	symmetric and radially decreasing, see for instance the main result of \cite{squ-acv}.
	Given $r_1$ and $r_2$ two solutions 
	to the minimization problem $\mathscr{M}(c)$ we have that $-\lambda_1$ and $-\lambda_2$ are the
	corresponding Lagrange multipliers and $\lambda_1,\lambda_2>0$ in light of \cite[Lemma 4.6]{CJS}
	\begin{equation}
		\label{changescal}
	- \Delta r_i- r_i\Delta r_i^2 +\lambda_i r_i=r_i^{p}\qquad\text{in $\R^3$},\qquad i=1,2.
	\end{equation}
	With respect to the semi-linear case, the main problem is the identification of Lagrange multipliers, which cannot be inferred as in the
	semi-linear case. In fact, let us consider as above the rescaling for $r_1$ by 
	$$
	w(x)=\mu r_1(\gamma x),\qquad \gamma:=\Big(\frac{\lambda_2}{\lambda_1}\Big)^\frac{1}{2},\quad
	\mu:=\Big(\frac{\lambda_2}{\lambda_1}\Big)^\frac{1}{p-1}.
	$$
	Then this yields $\|w\|_{L^2(\R^3)}^2=(\lambda_2/\lambda_1)^{\frac{4-3(p-1)}{2(p-1)}}c$ and
	\begin{equation*}
	- \Delta w- (\lambda_1/\lambda_2)^\frac{2}{p-1}w\Delta w^2 +\lambda_2 w=w^{p}\qquad\text{in $\R^3$}.
	\end{equation*}
	Although there are recent uniqueness results for the positive radial solutions to \eqref{changescal}, due to the presence
	of the residual coefficient $(\lambda_1/\lambda_2)^{2/(p-1)}$, we cannot infer as before that $w(\cdot)=r_2(\cdot+\xi)$
	for some point $\xi\in\R^3$. In the course of the next section, we shall compute the ground state  
under a conjectured property (indeed true in the semi-linear case discussed above) stated in the following
\begin{conject}
Assume that $r_1,r_2>0$ are radial decreasing solutions to \eqref{changescal} with $\lambda_1,\lambda_2>0$, 
${\mathcal E}(r_i)<0$  and $\|r_i\|^2_{L^2(\R^3)}=c$ for $i=1,2$. 
Then $\lambda_1=\lambda_2$ and $r_2=r_1(\cdot+\xi)$, for some $\xi\in\R^3$.
\end{conject}

\noindent
Consequently, given $c>0$, assume that there exist 
$r:\R^3\to\R$ and $\lambda>0$ such that
$$
\text{$r=\rho(|x|)>0$,\,\, $\rho'\leq 0$},\qquad
\|r\|^2_{L^2(\R^3)}=c,\qquad
{\mathcal E}(r)<0,
\qquad
- \Delta r- r\Delta r^2 +\lambda r=r^{p}\quad\text{in $\R^3$}.
$$
We know that this happens to be the case under the assumption that $\mathscr{M}(c)<0$.
Then $r$ is the unique solution to problem $\mathscr{M}(c)$, up to translations and multiplication by $\pm 1$.
\vskip2pt
\noindent
In figure~\ref{fig:confronto}, we compared the shape of the solutions to
\begin{equation}
	\label{alphamin}
\mathscr{M}_\vartheta(70):=\inf_{\substack{u\in X \\ \|u\|_{L^2(\R^3)}^2 = 70}} {\mathscr E}_\vartheta(u), \qquad
{\mathscr E}_\vartheta(u) := \frac{1}{2}\int_{\R^3} (1+2\vartheta u^2)|Du|^2dx- \frac{1}{3}\int_{\R^3} |u|^{3}dx.
\end{equation}
with $\vartheta=1$ (quasi-linear case) and $\vartheta=0$ (semi-linear case). Roughly speaking, the term 
$-u\Delta u^2$ produces an additional diffusive contribution which tends to squeeze 
the bump down against source effects.

\begin{figure}[!ht]
\includegraphics[scale=0.48]{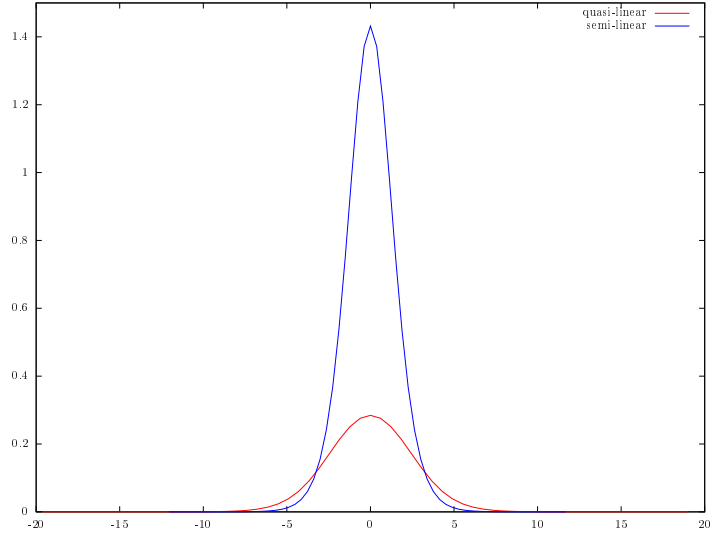}
\caption{Comparison between the solution (a section of the absolute
value squared) of~\eqref{alphamin} with $\vartheta=1$ (quasi-linear case)
and with $\vartheta=0$ (semi-linear case). The additional diffusive term present in the quasi-linear
case tends to produce squeezing effects.}
\label{fig:confronto}
\end{figure}

\section{Results}
Concerning Problem~\ref{bounds}, we have the following (see also Fig.~\ref{fig:bounds})

\begin{proposition}
	\label{variousnew}
	The following properties hold.
\begin{enumerate}
\item For every $\frac{7}{3} \leq p <\frac{13}{3}$ there holds
$$
\forall c>0:\quad
\mathscr{M}(c)\leq \inf_{\sigma> 0}\Big( \sigma^2\frac{3c}{4}+
\sigma^5\frac{3c^2}{8\sqrt{2}\pi^{3/2}}-
\sigma^{\frac{3(p-1)}{2}}\frac{2\sqrt{2}c^{\frac{p+1}{2}}\pi^{\frac{3}{2}-\frac{3(p+1)}{4}}}{(p+1)^\frac{5}{2}} \Big).
$$
\vskip4pt
\item For every $\frac{7}{3} \leq p <\frac{13}{3}$, if 
$$
{\mathcal K}_p:=4^{\frac{3p-3}{10}} S_{{\rm sob}}^{\frac{3p-3}{5}},\qquad S_{{\rm sob}}=\frac{2^{1/3}}{\sqrt{3\pi}}\frac{1}{\Gamma^{1/3}(3/2)},
$$
being $S_{{\rm sob}}$ the best Sobolev constant for the embedding of $H^1(\R^3)$ into $L^q(\R^3)$, there holds
$$
\forall c>0:\quad
\mathscr{M}(c)\geq -\frac{13-3p}{3(p-1)}\Big(\frac{10(p+1)}{3(p-1){\mathcal K}_p}\Big)^{\frac{10}{3p-13}} c^{\frac{11-p}{13-3p}}.
$$
\vskip4pt
\item
For $p=\frac{7}{3}$, setting $A:=\frac{2\sqrt{2}3^{5/2}}{10^{5/2}\pi}$,
$B:=\frac{3^{2/3}}{2^{7/3}\pi}$ and $c_\sharp:=(\frac{10^{5/2}\pi}{2^{7/2}3^{3/2}})^{3/2}$, there holds
$$
\forall c>0:\quad
\mathscr{M}(c)|_{p=\frac{7}{3}}\leq
\begin{cases}
	\,\,\, 0 & \text{if $c\leq c_\sharp$,}  \\
	-\big[\big(\frac{2}{5}\big)^{2/3}-\big(\frac{2}{5}\big)^{5/3}\big]\frac{[A c^{5/3}-3c/4]^{5/3}}{B c^{4/3}} & \text{if $c\geq c_\sharp$.}
\end{cases}
$$
In particular, $\mathscr{M}(c)|_{p=\frac{7}{3}}\apprle -Mc^{13/9},$ for every $c>0$ large and some $M>0$.
\item If $p=\frac{13}{3}$, setting 
$$
c_\flat:=(16/{3{\mathcal K}_{\frac{13}{3}}})^{3/2}\approx 19.73,\qquad {\mathcal K}_{\frac{13}{3}}:=4 S_{{\rm sob}}^2,
$$
we have $\mathscr{M}(c) =0$ for every $c \leq c_\flat$ and the infimum $\mathscr{M}(c)$ is not attained.
\vskip4pt
\item If $p=\frac{13}{3}$, setting 
$$
c^\flat:=\frac{3^{3/2}(16/3)^{15/4}\pi^{3/2}}{32^{3/2}}\approx 85.09,
$$ 
we have $\mathscr{M}(c) =-\infty$ for every $c>c^\flat$.
\end{enumerate}
\end{proposition}

\begin{proof}
Properties (1), (3) and (5) easily follow by the arguments in Section~\ref{numsec} and direct computations.
Properties (2) and (4) need bounds from below and can be justified as follows. The best Sobolev constant $S_{{\rm sob}}$ is computed
through the formula contained in \cite{talenti}. Concerning (4),
by H\"{o}lder and Sobolev inequalities, for $u \in X$ we have
\begin{equation}
	\label{sobcontroll}
   \int_{\R^3}|u|^{16/3} dx  \leq   \Big( \int_{\R^3} |u|^2 dx\Big)^{2/3}
   \Big( \int_{\R^3}|u|^{12}dx\Big)^{1/3} 
    \leq {\mathcal K}_{\frac{13}{3}} c^{2/3}\Big( \int_{\R^3}|u|^2 |Du|^2 dx \Big),
\end{equation}
where we have used the fact that ($2^*=6$ is the critical Sobolev exponent in $\R^3$)
$$
\int_{\R^3} |u|^{12}dx = \int_{\R^3}(u^2)^{2^*}dx,\qquad 
\int_{\R^3} |D(u^2)|^2dx = 4\int_{\R^3} |u|^2 |Du|^2 dx.
$$
From inequality~\eqref{sobcontroll} we infer
$$
{\mathcal E}(u) \geq \int_{\R^3}|u|^2|Du|^2 dx -
\frac{3{\mathcal K}_{\frac{13}{3}}}{16}c^{2/3} \int_{\R^3}|u|^2 |D u|^2dx,
$$
which yields ${\mathcal E}(u)\geq 0$ for every $u\in X$ and any $c\leq c_\flat$ and 
hence, in turn, the desired conclusion. In a similar fashion, concerning (2), if $p<13/3$,
using H\"{o}lder and Sobolev inequalities for any $u \in X$ we have
\begin{align}
   \label{1.1}
   \int_{\R^3}|u|^{p+1} dx  \leq  {\mathcal K}_p c^{\frac{11-p}{10}}
   \Big( \int_{\R^3}|u|^2 |\nabla u|^2 dx \Big)^{\frac{3p-3}{10}},
\end{align}
yielding immediately that
$$
{\mathcal E}(u) \geq \int_{\R^3}|u|^2|\nabla u|^2 dx -\frac{{\mathcal K}_p}{p+1} c^{\frac{11-p}{10}} \Big( \int_{\R^3}|u|^2 |\nabla u|^2
dx \Big)^{\frac{3p-3}{10}}.
$$
Since $p < 3/13$ it follows that $\frac{3p-3}{10}<1$. In turn, the function $\omega:[0,+\infty)\to\R$
$$
\omega(s)=s-\frac{{\mathcal K}_p}{p+1} c^{\frac{11-p}{10}} s^{\frac{3p-3}{10}},
\qquad s\geq 0,
$$
always admits a (negative) absolute minimum point at a point $t_{p,c}>0$ which can be easily computed,  
yielding the desired assertion by the arbitrariness of $u\in X$.
\end{proof}

\begin{figure}
\includegraphics[bb=0 359 577 793,scale=0.48]{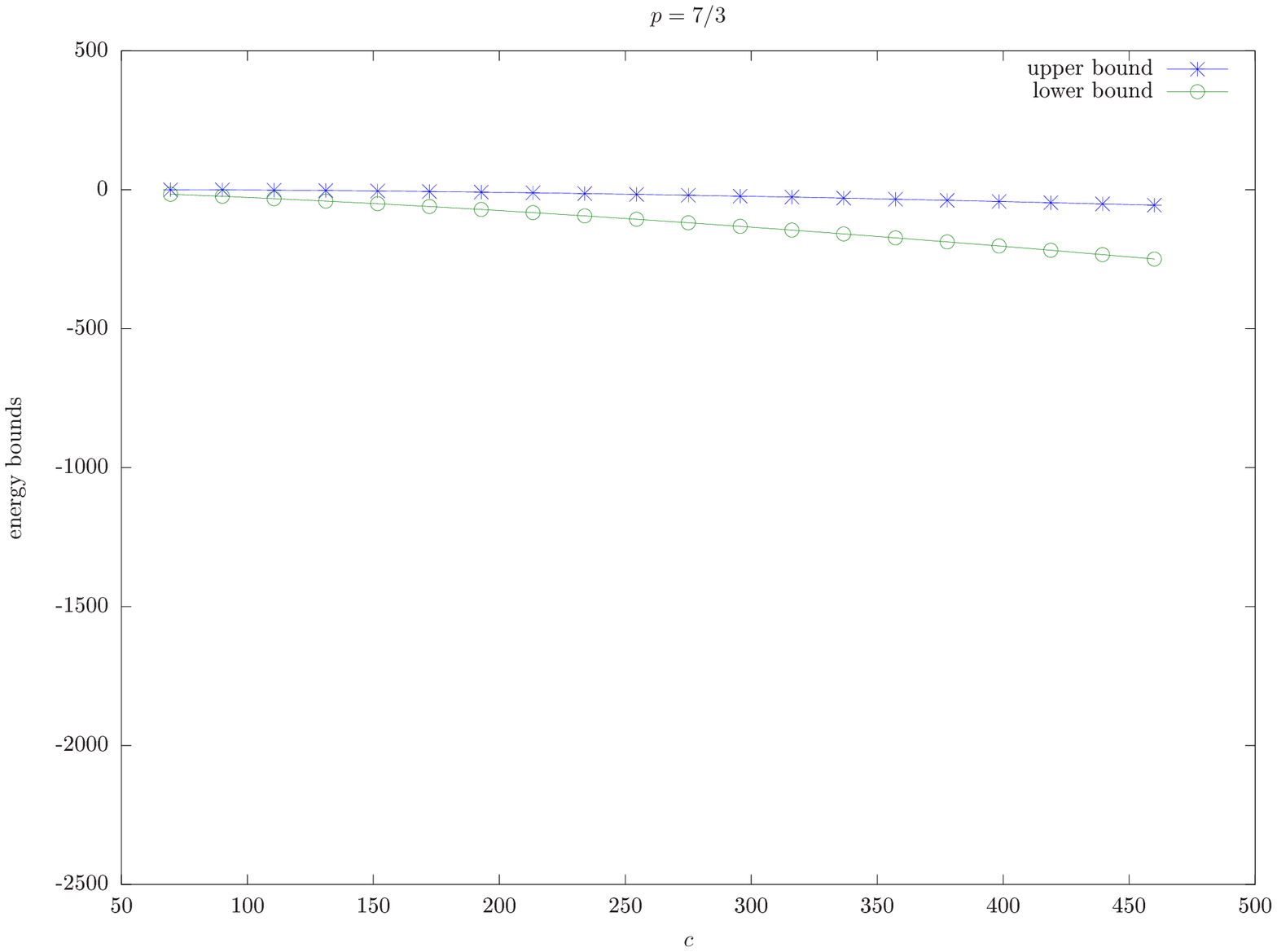}\\
\includegraphics[bb=0 359 577 793,scale=0.48]{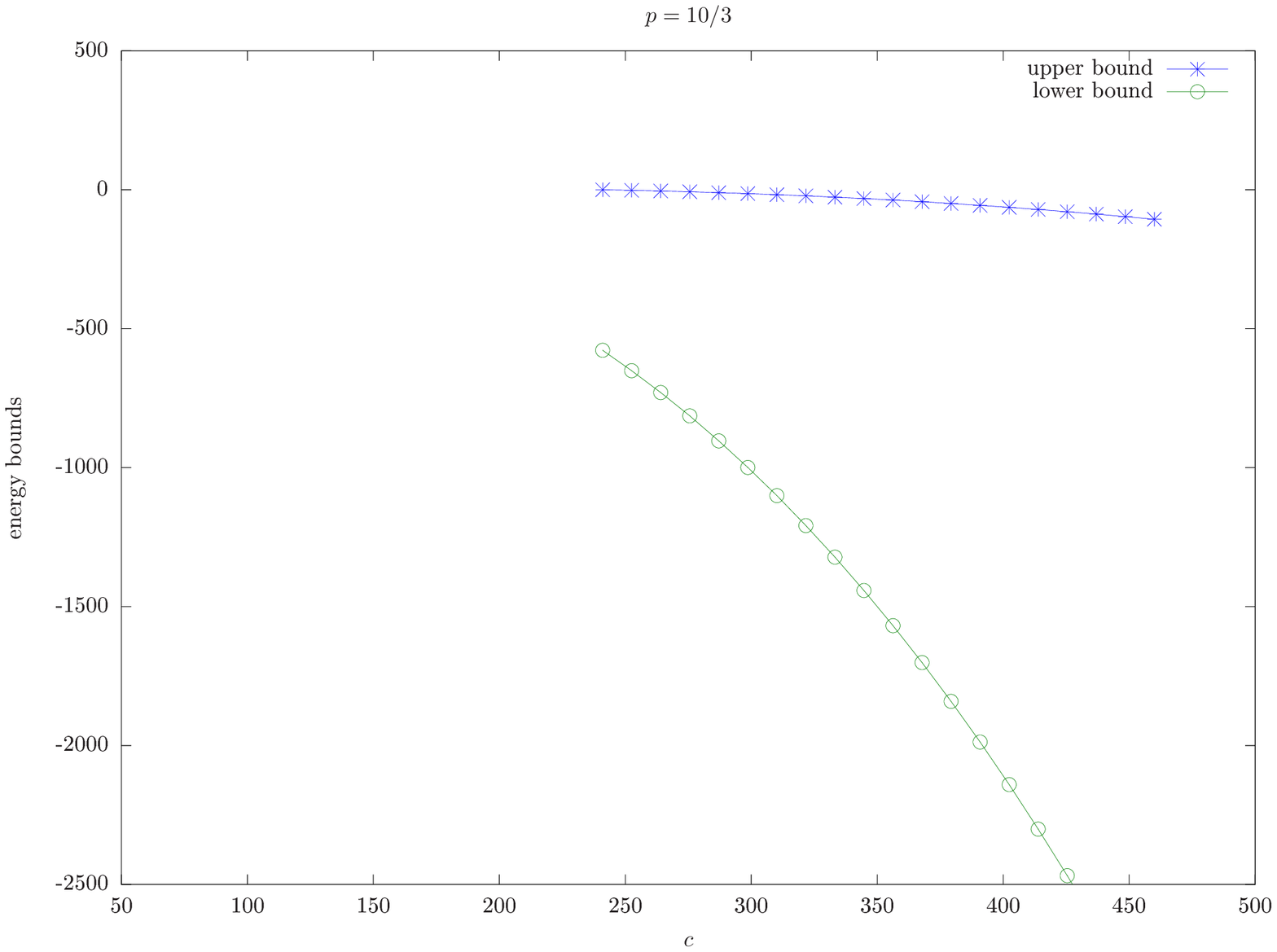}
\caption{Upper and lower bounds of ${\mathscr M}(c)$ according to the estimates obtained
Proposition~\ref{variousnew} in the particular cases $p=7/3$ (sharp bounds) and $p=10/3$ (less sharp bounds).}
\label{fig:bounds}
\end{figure}

\noindent
Concerning Problem~\ref{bifproblem} we shall provide an upper bounding profile for the values of $c_\sharp(p)$
and give indications showing that very likely $c_\sharp(p)$ remains in a small lower neighborhood of this
profile. As $p$ increases from $7/3$ up to a certain value $p_0\sim 3.3$ the bounding profile is increasing,
reaching values around $c=250$. Then, after $p_0$ it decreases. 
\vskip2.5pt
\noindent
Concerning Problem~\ref{gsproblem}, under the conjectured uniqueness result we compute the ground state solutions
for some values of $c$ greater than the upper bounding profile for the values of $c_\sharp(p)$. In the case $c<c_\sharp(p)$,
roughly speaking, if $(u_n)\subset X$ is an arbitrary minimizing sequence
for problem $\mathscr{M}(c)$, since we know that $\mathscr{M}(c)$ is not attained, $(u_n)$
cannot be strongly convergent, up to a subsequence, in $L^2(\R^3)$ and in turn in $H^1(\R^3)$. 
Then, by virtue of Lions's compactness-concentration principle, 
only {\em vanishing} or {\em dichotomy} might occur, in the language of \cite{pll-1,pll-2}. On 
the other hand, it was proved in \cite[Theorem 1.11]{CJS} that dichotomy can always be ruled out. In conclusion,
the only possibility remaining for a minimizing sequence is {\em vanishing}, precisely:
$$
\mbox{for every $R >0$:}
\qquad
\lim_{n\to\infty}\sup_{y\in\R^3}\int_{B(y,R)} u_n^2 dx=0,
$$
where $B(y,R)$ denotes the ball in $\R^3$ of center $y$ and radius $R$.
In particular, fixed any {\em bounded} domain $C$ in $\R^3$, imagined for instance as the computational domain,
the sequence $(u_n)$ {\rm cannot} be essentially supported into $C$, being for $R_0={\rm diam}(C)$
$$
\int_{\R^3\setminus C}u_n^2 dx\geq c-\sup_{y\in C}\int_{B(y,R_0)}u_n^2 dx\geq c-\sup_{y\in \R^3}\int_{B(y,R_0)}u_n^2 dx\geq c/2,
$$
for $n$ sufficiently large. This means that $(u_n)$ (in particular any numerically approximating ${\mathscr M}(c)$) 
tends to spread out of {\em any} fixed (computational) domain $C$.
For instance, the sequence $(g_{c,\sigma_j})$ with $\sigma_j\to 0$ as $j\to\infty$, where $g_{c,\sigma}$ is defined in Section~\ref{numsec}
is such that $\|g_{c,\sigma_j}\|_{L^2(\R^3)}^2=c$ and ${\mathcal E}(g_{c,\sigma_j})\to 0$ as $j\to\infty$ and hence, for $c<c_\sharp(p)$,
being ${\mathscr M}(c)=0$, it is a minimization sequence. It shows vanishes, since 
$$
\sup_{y\in\R^3}\int_{B_R(y)} g_{c,\sigma_j}^2dx\leq CR^3\sigma_j^3,
$$ 
for all $R >0$ and $j\geq 1$.

\section{Numerical approximation}\label{numsec}
Instead of a direct minimization of the energy functional~\eqref{eq:energyfunc}
(see, for instance, \cite{CORT09}),
as seen in \cite{BD04,CS}, 
it is also possible to find a solution of~\eqref{eq:stationary}
by solving the parabolic problem
\begin{equation}\label{eq:parabolic}
\left\{
\begin{aligned}
\partial_t u&=\Delta u+u\Delta u^2+|u|^{p-1}u-\lambda(u)u\\
u(0)&=u_0,\quad \|u_0\|^2_{L^2(\R^3)}=c
\end{aligned}\right.
\end{equation}
up to the steady-state $u_c$, where $\lambda(u)$ is defined by
\begin{equation*}
\lambda(u)=-\frac{\int_{\R^3}(1+4u^2)|Du|^2-\int_{\R^3}|u|^{p+1}}{\int_{\R^3} u^2}
\end{equation*}
This approach in known as continuous normalized gradient flow.
It is easy to show that the energy associated to the solution $u(t)$ 
decreases in time, whereas the $L^2$-norm is constant.

\begin{figure}
\includegraphics[bb=0 359 577 793,scale=0.48]{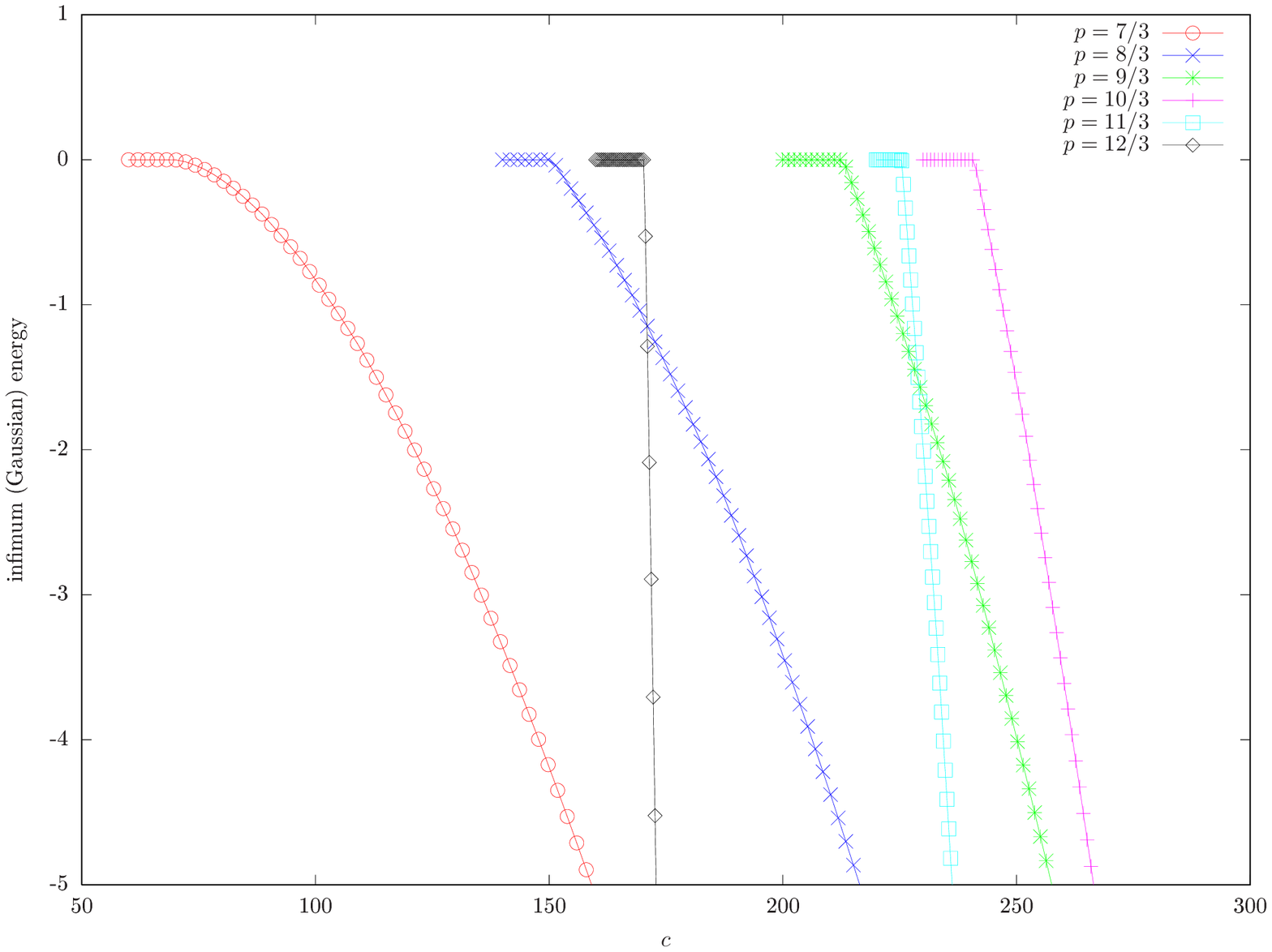}\\
\includegraphics[bb=0 359 577 793,scale=0.48]{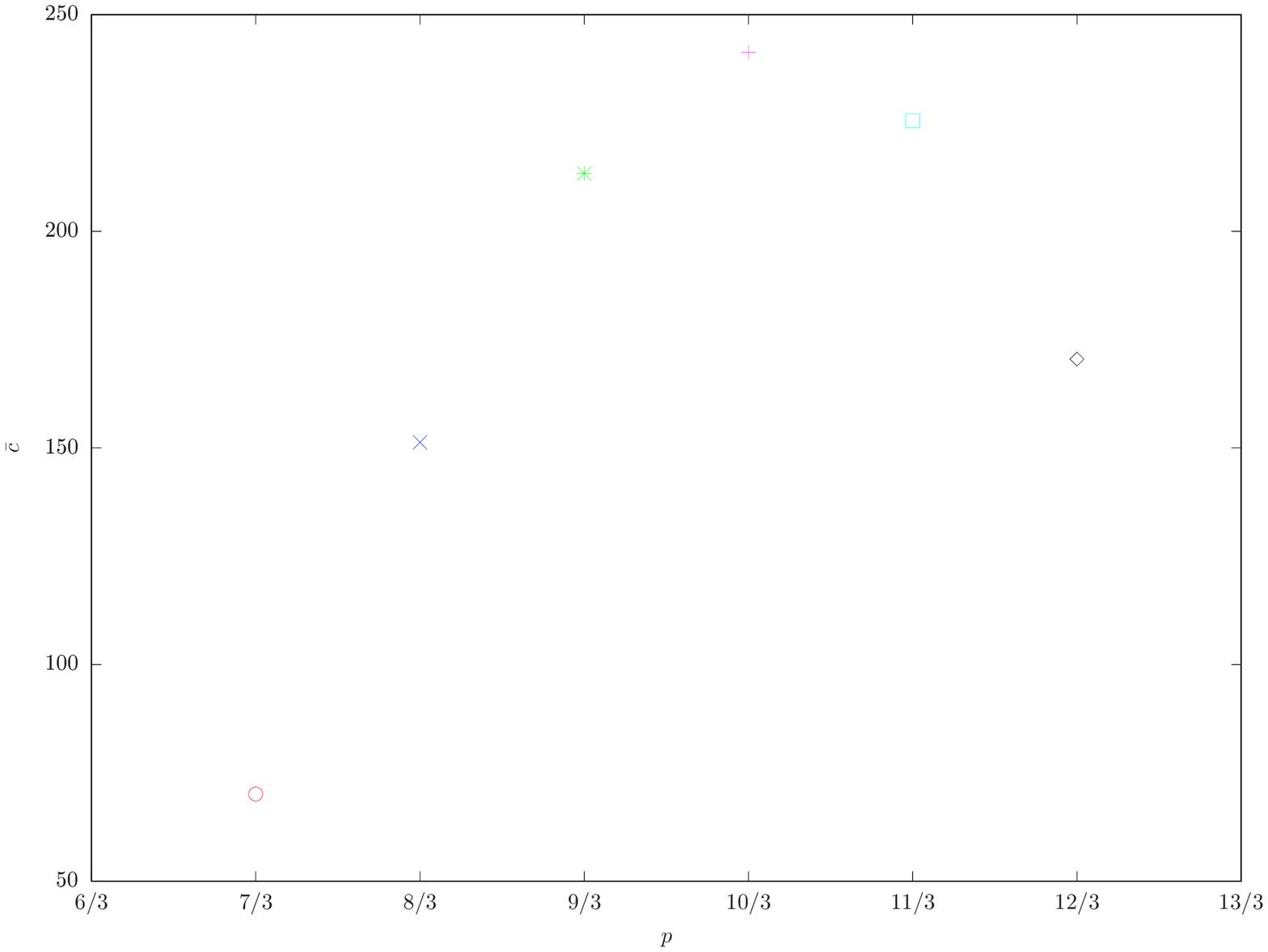}
\caption{Values of the infimum (Gaussian) energy with respect to $c$
(left) and values of $c_\mathrm{g}$ (right): from $p=7/3$ to $p=12/3$ they
are $69.510$, $150.65$, $212.91$, $241.01$, $225.42$, $170.33$, respectively.}
\label{fig:energy_gaussian}
\end{figure}

In order to find a ``good'' initial solution $u_0$, that is a function
with a negative energy, 
we can consider the family of Gaussian radial functions
\begin{equation*}
g_{c,\sigma}(r)=\frac{\sqrt{c\sigma^3}}{\sqrt[4]{\pi^3}}e^{-\sigma^2 r^2/2},\quad
\|g_{c,\sigma}\|^2_{L^2(\R^3)}=c,\quad \sigma>0
\end{equation*}
and minimize the energy $\mathscr {E}(g_{c,\sigma})$, which, for a 
given $p$ and $c$, can be computed 
analytically as 
\begin{equation}\label{eq:energygaussian}
\mathscr {E}(g_{c,\sigma})=\sigma^2\frac{3c}{4}+
\sigma^5\frac{3c^2}{8\sqrt{2}\pi^{3/2}}-
\sigma^{\frac{3(p-1)}{2}}\frac{2\sqrt{2}c^{\frac{p+1}{2}}\pi^{\frac{3}{2}-\frac{3(p+1)}{4}}}{(p+1)^\frac{5}{2}}
\end{equation}
with respect to the parameter $\sigma$. If the infimum value for the energy 
is zero,
we increase the value of $c$ and look again for the infimum energy. We proceed
with increasing values of $c$ until we find a $c_\mathrm{g}$ such that the
minimum energy with respect to $\sigma$, corresponding
to a value $\bar \sigma$, is negative. This is possible, since we
consider a discrete sequence of increasing values for $c$ in order
to test the negativity of the energy.
Such a 
$c_\mathrm{g}$ is clearly an upper bound for the desired value $c_\sharp$. For
the range $p=7/3,8/3,9/3,10/3,11/3,12/3$ we obtain the values reported
in Figure~\ref{fig:energy_gaussian} (right).

Now, we are ready to look for values $c<c_\mathrm{g}$, 
for which the steady-state
of~\eqref{eq:parabolic} has a negative energy. To this purpose,
we choose $u_0=g_{c,\sigma}$, where $\sigma = \bar \sigma \cdot c/c_\mathrm{g}$. 
The meaning of this choice is
the following: since $c<c_\mathrm{g}$, the infimum value of the
energy attained by a Gaussian
function is zero and corresponds to the limit case $\sigma\to0$. 
We instead select 
$0<\sigma< \bar \sigma$ with the idea that $g_{c,\sigma}$ is a good 
initial value, because close to
$g_{c_\mathrm{g},\bar \sigma}$, that is the optimal element in the Gaussian family. 
With this choice, clearly we have $\mathscr{E}(u_0)>0$.

In order to fix once and for all the computational domain in such a way that
it does not depend on $\sigma$, 
we scale the space variables by $\sigma$ and
the unknown $u$ in order to have unitary $L^2$ norm, that is
\begin{equation*}
\sqrt{c\sigma^3}v_c(t,\sigma \cdot)=u(t,\cdot)
\end{equation*}
We end up with
\begin{equation}\label{eq:parabolicscaled}
\left\{
\begin{aligned}
\partial_t v_c&=\sigma^2\Delta v_c+c\sigma^5v_c\Delta v_c^2+
(c\sigma^3)^{\frac{p-1}{2}}|v_c|^{p-1}v_c+\eta(v_c)v_c\\
v_c(0)&=\frac{1}{\sqrt[4]{\pi^3}}e^{-r^2/2}
\end{aligned}\right.
\end{equation}
where
\begin{equation*}
\eta(v_c)=\frac{\int_{\R^3}\sigma^2(1+4c\sigma^3v_c^2)|Dv_c|^2-
(c\sigma^3)^{\frac{p-1}{2}}\int_{\R^3}|v_c|^{p+1}}{\int_{\R^3} v_c^2}
\end{equation*}
The corresponding energy is
\begin{equation*}
\mathscr{E}(u)=E(v_c)=\frac{1}{2}
\int_{\R^3}c \sigma^2(1+2c\sigma^3v_c^2)|Dv_c|^2-
\frac{c(c\sigma^3)^{\frac{p-1}{2}}}{p+1}\int_{\R^3}|v_c|^{p+1}.
\end{equation*}
We solve equation~\eqref{eq:parabolicscaled} up to a final time $T$ for which
\begin{equation*}
\|\bar \sigma^2\Delta v_c+c\bar \sigma^5v_c\Delta v_c^2+
(c\bar \sigma^3)^{\frac{p-1}{2}}|v_c|^{p-1}v_c+\eta(v_c)v_c\|_{L^2(\R^3)}<
\mathrm{tol}
\end{equation*}
where $\mathrm{tol}$ is a prescribed tolerance to detect the approximated
steady-state. 
As already done in~\cite{CS}, we apply the exponential Runge--Kutta method
of order two (see~\cite{HO10}) to the spectral Fourier decomposition in space 
of~\eqref{eq:parabolicscaled}. The embedded exponential Euler method gives
the possibility to derive a variable stepsize integrator, which is particularly
useful when approaching the steady-state solution, allowing the time steps
to become larger and larger. For our numerical experiments, we used 
the computational domain $[-5,5]^3$, the regular grid of $64^3$ points, a
tolerance for the local error (in the $L^2$ norm) 
equal to $10^{-8}$ and the steady-state detection tolerance 
$\mathrm{tol}=10^{-7}$.

The solution $v_c(T)$ is
then considered an approximated steady-state solution. If its energy is negative
and it is radially symmetric and decreasing, then we conclude it is a minimum
for $\mathscr{M}(c)$. Therefore, $c\ge c_\sharp$ is the current upper 
bound for $c_\sharp$ (blue circle in Figure~\ref{fig:stati}). 

On the other hand, if $c<c_\sharp$, then $\mathscr{M}(c)=0$ and the
infimum is approximated by flatter and flatter functions. This is
perfectly clear in the case one restricts the search among Gaussian
functions, for which $\mathscr{E}(g_{c,\sigma})\to0$ for $\sigma\to0$.
This situation can be recognized in the numerical experiments because
the essential support of $v_c(t)$ during time evolution tends to grow and
to spread out the computational domain (green plus in Figure~\ref{fig:stati}), 
which was chosen in such a way
to comfortably contain the essential support of the initial solution
$v_c(0)$. In this case,
we apply a bisection algorithm on the values $c$ and $c_\mathrm{g}$ 
(the Gaussian upper bound for $c_\sharp$) in order to
find a tighter upper bound for $c_\sharp$, since $c<c_\sharp$ and 
$c_\mathrm{g}\ge c_\sharp$.

\begin{figure}
\includegraphics[scale=0.48]{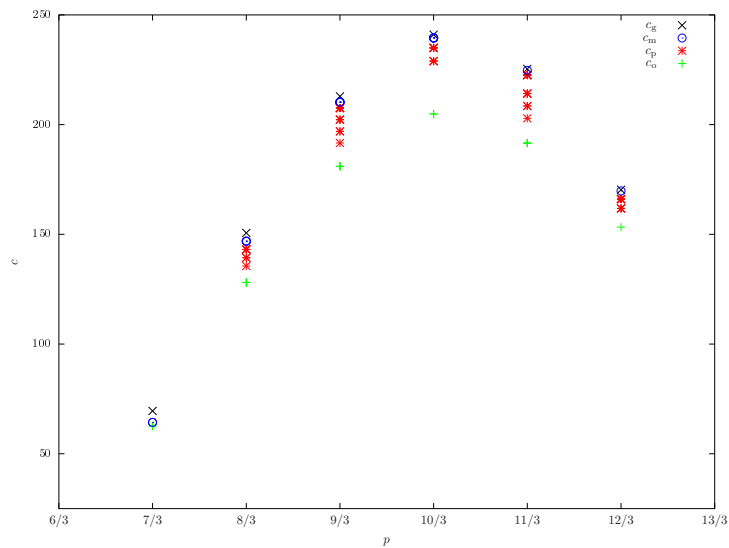}
\caption{Values of $c_\mathrm{g}$ (Gaussian upper bound of $c_\sharp$), 
$c_\mathrm{m}$ (values
for which $E(v_{c_\mathrm{m}}(T))<0$) and $c_\mathrm{p}$ (values for which 
$E(v_{c_\mathrm{p}}(T))>0$). For instance, for the case $p=9/3$ 
and $c=0.9875\cdot c_\mathrm{g} \approx 210.25$ 
we found an approximated steady-state 
with energy $E(v_c(T))\approx -1.30\cdot 10^{-1}$ (blue circle), whereas
with $c=0.975\cdot c_\mathrm{g}\approx 207.59$, 
we found an approximated steady-state with energy 
$E(v_c(T))\approx 6.97\cdot 10^{-2}$ (red star).
}
\label{fig:stati}
\end{figure}

In the numerical experiments we encountered another situation: starting
from an initial solution with positive energy, it was possible to find a
radially symmetric and decreasing approximated steady-state solution, 
whose support
was perfectly contained into the computational domain and with a 
positive
energy (red star in Figure~\ref{fig:stati}). 
This solution is not a solution with minimum energy, because
for $\sigma$ small enough it is possible to find a Gaussian function
$g_{c,\sigma}$ with smaller energy. 
In our numerical experiments, we
observed this behaviour, for a given $p$ and with the tolerances
described above, 
for the values of $c$ between the first value for which the essential 
support of the
solution spread out the computational domain and the current upper
bound for $c_\sharp$.
We notice that we were not able to obtain
solutions with positive energy in the limit case $p=7/3$.

Overall, from the numerical experiments we found out that the higher is 
$p$ the more difficult is to find a value of $c<c_\mathrm{g}$ for
which the steady-state has a negative energy and can thus be considered a 
solution of minimum energy.

\bigskip
\bigskip

\end{document}